\documentclass[12pt]{article}

\usepackage{setspace}

\usepackage{amsthm, amssymb, amstext}
\usepackage[fleqn]{amsmath}
\usepackage{latexsym}
\usepackage[dvips]{graphicx}
\usepackage{comment}
\usepackage{hyperref}
\usepackage{mathtools}
\usepackage{enumerate} 

\usepackage{todonotes}
\usepackage{comment}

\def\A{\mathcal{A}}
\def\B{\mathcal{B}}

\def\F{\mathcal{F}}

\def\M{\mathcal{M}}
\def\N{\mathcal{N}}
\def\S{\mathcal{S}}

\newtheorem{theorem}{Theorem}[section]

\newtheorem{claim}{Claim}
\newtheorem{problem}{Problem}

\theoremstyle{definition}
\newtheorem{definition}{Definition}[section]

\usepackage{amsthm}
\numberwithin{equation}{section}

\title{Intersecting families, signed sets, and injection}
\date{}
\author{Carl Feghali \thanks{Computer Science Institute, Faculty of Mathematics and Physics, Charles University, Prague, Czech Republic \newline email: \texttt{feghali.carl@gmail.com} }}

\begin{document} 

\maketitle  

\begin{abstract}
 Let $k, r, n \geq 1$ be integers, and let $\S_{n, k, r}$ be the family of $r$-signed $k$-sets on $[n] = \{1, \dots, n\}$ given by
$$
\S_{n, k, r} = \Big\{\{(x_1, a_1), \dots, (x_k, a_k)\}: \{x_1, \dots, x_k\} \in \binom{[n]}{k}, a_1, \dots, a_k \in [r] \Big\}.
$$
A family $\A \subseteq \S_{n, k, r}$ is \emph{intersecting} if $A, B \in \A$ implies $A \cap B \not= \emptyset$. A well-known result (first stated by Meyer and proved using different methods by Deza and Frankl, and Bollob\'as and Leader) states that if $\A \subseteq \S_{n, k, r}$ is intersecting, $r \geq 2$ and  $1 \leq k \leq n$, then $$|\A| \leq r^{k-1}\binom{n-1}{k - 1}.$$
We provide a  proof of this result by injection (in the same spirit as Frankl and F\"uredi's and Hurlbert and Kamat's injective proofs of the Erd\H{o}s--Ko--Rado Theorem, and  Frankl's and Hurlbert and Kamat's injective proofs of the Hilton--Milner Theorem)  whenever $r \geq 2$ and $1 \leq k \leq n/2$, leaving open only some cases when $k \leq n$. 

\end{abstract}

\section{Introduction}

Let $[n] = \{1, \dots, n\}$ and let $\binom{[n]}{k}$ denote the collection of all $k$-subsets of $[n]$. Sets of sets are called \emph{families}. A family $\F \subseteq 2^{[n]}$ is \emph{intersecting} if $F, F' \in \F$ implies $F \cap F' \not= \emptyset$. How large can an intersecting family $\F \subseteq \binom{[n]}{k}$ be? If $2k > n$ then $|\F| = \binom{n}{k}$ is obvious, while if $2k \leq n$ then the answer is given by the classical Erd\H{o}s--Ko--Rado Theorem \cite{erdos}. 

\begin{definition}
Let
$$
\S = \Big\{F \in \binom{[n]}{k}: 1 \in F\Big\}.
$$
\end{definition}

\bigskip

\noindent
\textbf{Erd\H{o}s--Ko--Rado Theorem} (Erd\H{o}s, Ko and Rado \cite{erdos}).
\textit{Let $n, k \geq 0$ be integers, $n \geq 2k$. Let $\F \subseteq \binom{[n]}{k}$ be intersecting. 
Then }
\begin{equation}\label{eq:ekr} |\F| \leq \binom{n - 1}{k - 1} = |\S|.\end{equation}

\bigskip

When $n = 2k$, the proof of the Erd\H{o}s--Ko--Rado Theorem is easy. Simply partition $\binom{[2k]}{k}$ into complementary pairs. Then, since $\F$ can contain at most one set from each pair, $|\F| \leq \frac{1}{2}\binom{2k}{k} = \binom{2k - 1}{k - 1}$. To deal with the case $n > 2k$ Erd\H{o}s, Ko and Rado \cite{erdos} introduced an important operation on families called \emph{shifting}.  

A family is called \emph{non-trivial} if there is no element common to all its members. Hilton and Milner \cite{hilton} showed that for $n > 2k$, $\S$ is the unique maximal intersecting family. 

\begin{definition}
Let $G \in \binom{[n]}{k}$, $1 \not\in G$  and 
$$
\N = \{G\} \cup \Big\{F \in \binom{[n]}{k}: 1 \in F, F \cap G \not= \emptyset\}.
$$
\end{definition}

\bigskip

\noindent
\textbf{Hilton--Milner Theorem} (Hilton and Milner \cite{hilton}).
\textit{Suppose that $n \geq 2k \geq 4$ and $\F \subseteq \binom{[n]}{k}$ is non-trivial. Then }
\begin{equation}\label{eq:hm}
|\F| \leq \binom{n - 1}{k - 1} - \binom{n - k - 1}{k - 1} + 1 = |\N|. 
\end{equation} 
\bigskip

There are various proofs of the Erd\H{o}s--Ko--Rado Theorem (cf.\ \cite{ daykin, frankl, hurlbertnew, katonaekr}) and the Hilton--Milner Theorem (cf.\ \cite{franklsimple, franklhm, hurlbertnew}).  To keep this paper short, let us highlight those which are particularly relevant to us: Frankl and F\"uredi's \cite{frankl} and Hurlbert and Kamat's \cite{hurlbertnew} injective proofs of (\ref{eq:ekr}), and Frankl's \cite{franklsimple} and Hurlbert and Kamat's \cite{hurlbertnew} injective proofs of (\ref{eq:hm}). 

We should mention that by ``injective proof" we mean an explicit or implicit injection from $\mathcal{F}$ into a given intersecting family (usually a family whose members contain a prescribed element). We believe that such proofs are of interest, particularly in yielding further insight for the cases when the size of intersecting families cannot be determined \emph{a priori}; as an example of such a case see \cite[Conjecture 1.4]{borgsigned1}.   For further results in extremal set theory, we refer the reader to the excellent monograph by Gerbner and Patkos \cite{gerbner}. 

We now define \emph{signed sets}.  Let $k, r, n \geq 1$ be integers, and let $\S_{n, k, r}$ be the family of $r$-signed $k$-sets on $[n]$ given by
$$
\S_{n, k, r} = \Big\{\{(x_1, a_1), \dots, (x_k, a_k)\}: \{x_1, \dots, x_k\} \in \binom{[n]}{k}, a_1, \dots, a_k \in [r] \Big\}.
$$

A well-known analogue of the Erd\H{o}s--Ko--Rado Theorem for signed sets was first stated by Meyer \cite{meyer}, and later proved by Deza and Frankl \cite{deza} using the shifting technique, and by Bollob\'as and Leader \cite{leader} using Katona's elegant cycle method \cite{katonaekr}. 

\begin{definition}
Let 
$$
\mathcal{W} = \Big\{ W \in \S_{n, k, r}: (1, 1) \in W\Big\}. 
$$
\end{definition}

\begin{theorem}[Deza and Frankl \cite{deza}; Bollob\'as and Leader \cite{leader}]\label{thm:signed}
Let $n, k, r \geq 1$ be integers, $n \geq k$. Let $\F \subseteq \S_{n, k, r}$ be intersecting. Then
\begin{equation}\label{eq:signed}
|\F| \leq r^{k - 1}\binom{n - 1}{k - 1} = |\mathcal{W}|. 
\end{equation}
\end{theorem}

\bigskip

We should mention that there are several generalisations, extensions and variations of Theorem \ref{thm:signed}; see for example \cite{ahlswede, bey, borgsigned1,  borghilton, borgmultiple, engel1, livingstone}. 

Motivated by the afore-mentioned results we consider the following problem. 

\begin{problem}\label{problem}
Find an injective proof of (\ref{eq:signed}). 
\end{problem}

The object of this paper is to present the following theorem that provides extensive solutions to Problem \ref{problem} leaving open  only some cases when $k \leq n$. 

\begin{theorem}\label{thm:main}
There is an injective proof of (\ref{eq:signed}) whenever $r \geq 2$ and $k \leq n / 2$. 
\end{theorem}


\section{The proof}

One of the main tools in our proof is Katona's \emph{Intersection Shadow Theorem}. For integers $k > s \geq 0$ and a family $\F \subseteq \binom{[n]}{k}$, define its $s$-shadow $\partial_s(\F)$ by 
$$
\partial_s(\F):=\Big\{ G \in \binom{[n]}{s}: \exists F \in \F, G \subset F\Big\}.
$$

Suppose that  $\F \subseteq \binom{[n]}{s}$ such that $|F \cap F'| \geq  t \geq 0$ for all $F, F' \in \F$. Katona \cite{katona} then showed that
\begin{equation}\label{katona}
|\partial_{s - t}(\F)| \geq |\F|. 
\end{equation}

Let $\mbox{mod}^*$ be the usual modulo operation except that for integers $x$ and $y$, $(xy) \mbox{ mod}^* y$ is $y$ instead of $0$. Following Borg \cite{borgsigned1}, for a signed sets $A$ and integers $q$ and $r$, let $\theta^q_r(A)$ be the shifting operation given by
$$
\theta^q_r(A) = \{(x, (a + q) \mbox{ mod}^* r): (x, a) \in A\},
$$
and, for a family $\A$ of signed sets, $$\theta^q_r(\A) = \{\theta^q_r(A): A \in \A\}.$$

\begin{proof}[Proof of Theorem \ref{thm:main}] The proof consists of a nearly straightforward combination of the arguments found in \cite{borgfeghalicycle, frankl}. Let $\A \subseteq \S_{n, k, r}$ be intersecting, let $\A_0 = \{A \in \A: A \cap \{(1, 1), \dots, (1, r)\} = \emptyset\}$ and $\A_i = \{A \in \A: (1, i) \in A\}$ for $1 \leq i \leq r$. Note that $\A_0, \dots, \A_r$ partition $\A$. Let $\A'_0 = \A_0$ and $\A'_i = \{A \setminus \{(1, i)\}: A \in \A_i\}$ for $1 \leq i \leq r$.

Let $\A' = \bigcup_{i = 0}^r \A'_r$. For $A \in \A'$, let $M_A = \{x: (x, a) \in A \}$. We say that $M_A$ \emph{represents} $A$. Let $\mathcal{M}_0 = \{M_A: A \in \A'_0\}$, $\mathcal{M}_1 = \{M_A: A \in \A' \setminus \A'_0\}$, $\N = \{[2, n] \setminus  M: M \in \mathcal{M}_0\}$ and
$$
\B = \Big\{(x_1, a_1), \dots, (x_{k - 1}, a_{k - 1}) : \{x_1, \dots, x_{k - 1}\} \in \partial_{k - 1}(\N), a_1, \dots, a_{k -1} \in [r]\Big\}.
$$

\begin{claim}\label{claim:2}
$|\A'_0| \leq |\B|$.
\end{claim}

\noindent
\textit{Proof.}
 Since $\A'_0$ is intersecting, 
\begin{equation}\label{eq:intersect} \mbox{each set in $\mathcal{M}_0$ can represent at most $r^{k - 1}$ sets in $\A'_0$.} \end{equation}

Let $N, N' \in \mathcal{N}$. Since $1 \leq k \leq n/2$, we infer $$
|N \cap N'| = |([2, n] \setminus M) \cap ([2, n] \setminus M')| = n - 1 - 2k + |M \cap M'| \geq n - 2k \geq 0,
$$
so that applying (\ref{katona}) with  $s = n - 1 - k$ and $t = n - 2k$ gives us
\begin{equation}\label{eq:bound}
|\M_0| = |\N| \leq |\partial_{k - 1}(\N)|.
\end{equation}
Then (\ref{eq:intersect}) and (\ref{eq:bound}) yield
\begin{equation*}\label{eq:fin}
|\A'_0| \leq r^{k - 1}|\M_0| \leq  r^{k-1}|\partial_{k-1}(\N)| =  |\B|. \qed
\end{equation*}

\begin{claim}\label{claim}
The families $\A'_1, \theta^1_r(\A'_2), \dots, \theta^{r - 1}_r(\A'_{r}), \B$ are pairwise disjoint. 
\end{claim}

\begin{proof}
Since $\A$ is intersecting, 
\begin{equation}\label{eq:cross}
\mbox{for $i, j \in \{0\} \cup [r]$ with $i \not=j$ each set in $\A'_i$ intersects  each set in $\A'_j$.} 
\end{equation}

Suppose there exists $B \in \theta^{i - 1}_r(\A'_i) \cap \theta^{j - 1}_r(\A'_j)$ for some distinct $i, j \in  [2, r]$. Let $A_i = \theta^{-(i - 1)}_r(B) \in \A'_i$ and $A_j = \theta^{-(j - 1)}_r(B) \in \A'_j$. Then $A_i \cap A_j = \emptyset$, which contradicts~(\ref{eq:cross}). 
 Similarly, if we suppose $B \in \A'_1 \cap \theta^{i - 1}_r(\A'_i)$ for some $i \in [2, r]$, then we get a contradiction to (\ref{eq:cross}). Therefore, families $\A'_1, \theta^1_r(\A'_2), \dots, \theta^{r - 1}_r(\A'_{r})$ are pairwise disjoint. By (\ref{eq:cross}), 
each set in $\M_0$ intersects each set in $\M_1$. Therefore $\M_1 \cap \partial_{k - 1}(\N) = \emptyset$, which is to say
$$ \B\cap \bigg( \A'_1 \cup \bigcup_{i = 2}^{r} \theta^{i - 1}_r(\A'_i)\bigg) = \emptyset
$$ and the claim is proved. 
\end{proof}

Let $\A^*_0 = \{B \cup \{(1, 1)\} \colon  B \in \B\}$, $\A_1^* = \A_1$ and $\A^*_i = \{A \cup \{(1, 1)\} \colon A \in \theta^{i - 1}_r(\A'_i) \}$ for $2 \leq i \leq r$. For $0 \leq i \leq r$, $\A^*_i \subseteq \mathcal{W}$. By Claim~\ref{claim}, $\sum_{i = 0}^p |\A^*_i| \leq |\mathcal{W}|$. By Claim \ref{claim:2},
$|\A_0| \leq |\A^*_0|$. We have 
$$|\A| = \sum_{i = 0}^r |\A_i| = |\A_0| + \sum_{i = 1}^{r } |\A_i^*| \leq \sum_{i = 0}^r |\A^*_i| \leq |\mathcal{W}|,$$
and the theorem is proved. 
\end{proof}

\section*{Acknowledgements}

The author was partially supported by grant 
249994 of the research Council of Norway via the project CLASSIS and by grant 19-21082S of the Czech Science Foundation.


\begin{thebibliography}{}


\bibitem{Woodroofe} R. Woodroofe, Erd\H os--Ko--Rado theorems for simplicial complexes, J. Combin. Theory Ser. A 118 (2011), 1218--1227.

\bibitem{ahlswede}
R. Ahlswede and L. H. Khachatrian, The diametric theorem in Hamming spaces--optimal anticodes, 
 Advances in Applied Mathematics, 20(4):429--449, 1998.

\bibitem{bey}
C.~Bey.
\newblock An intersection theorem for weighted sets.
\newblock {\em Discrete Mathematics}, 235(1-3):145--150, 2001.

\bibitem{leader}
B.~Bollob{\'a}s and I.~Leader.
\newblock An {E}rd{\H{o}}s-{K}o-{R}ado theorem for signed sets.
\newblock {\em Computers \& Mathematics with Applications}, 34(11):9--13, 1997.

\bibitem{borgsigned1}
P.~Borg.
\newblock Intersecting systems of signed sets.
\newblock {\em the electronic journal of combinatorics}, 14(1):41, 2007.

\bibitem{borghilton}
P.~Borg.
\newblock A {H}ilton--{M}ilner-type theorem and an intersection conjecture for
  signed sets.
\newblock {\em Discrete Mathematics}, 313(18):1805--1815, 2013.

\bibitem{borgfeghalicycle}
  P.~Borg and C.~Feghali. 
  \newblock{On the {H}ilton--{S}pencer intersection theorems for unions of cycles},
  \newblock{\em arXiv preprint} arXiv:1908.08825, 2019. 
 

\bibitem{borgmultiple}
P.~Borg and I.~Leader.
\newblock Multiple cross-intersecting families of signed sets.
\newblock {\em Journal of Combinatorial Theory, Series A}, 117(5):583--588,
  2010.

\bibitem{daykin}
D.~E. Daykin.
\newblock {E}rd{\"o}s-{K}o-{R}ado from {K}ruskal-{K}atona.
\newblock {\em Journal of Combinatorial Theory, Series A}, 17(2):254--255,
  1974.

\bibitem{deza}
M.~Deza and P.~Frankl.
\newblock {Erd\H{o}s-Ko-Rado} theorem--22 years later.
\newblock {\em SIAM Journal on Algebraic Discrete Methods}, 4(4):419--431,
  1983.

\bibitem{engel1}
K.~Engel.
\newblock An {E}rd{\"o}s-{K}o-{R}ado theorem for the subcubes of a cube.
\newblock {\em Combinatorica}, 4(2-3):133--140, 1984.

\bibitem{erdos}
P.~Erd\H{o}s, C.~Ko, and R.~Rado.
\newblock Intersection theorems for systems of finite sets.
\newblock {\em The Quarterly Journal of Mathematics}, 12(1):313--320, 1961.

\bibitem{franklsimple}
P.~Frankl.
\newblock A simple proof of the {H}ilton--{M}ilner theorem.
\newblock {\em Moscow Journal of Combinatorics and Number Theory},
  8(2):97--101, 2019.

\bibitem{franklhm}
P.~Frankl and Z.~F{\"u}redi.
\newblock Non-trivial intersecting families.
\newblock {\em Journal of Combinatorial Theory, Series A}, 41(1):150--153,
  1986.

\bibitem{frankl}
P.~Frankl and Z.~F{\"u}redi.
\newblock A new short proof of the {EKR} theorem.
\newblock {\em Journal of Combinatorial Theory, Series A}, 119(6):1388 -- 1390,
  2012.

\bibitem{gerbner}
D.~Gerbner and B.~Patk{\'o}s.
\newblock {\em Extremal Finite Set Theory}.
\newblock CRC Press, 2018.

\bibitem{hilton}
A.~J.~W. Hilton and E.~C. Milner.
\newblock Some intersection theorems for systems of finite sets.
\newblock {\em The Quarterly Journal of Mathematics}, 62(3):625--635, 1967.

\bibitem{hurlbertnew}
G.~Hurlbert and V.~Kamat.
\newblock New injective proofs of the {E}rd{\H{o}}s--{K}o--{R}ado and
  {H}ilton--{M}ilner theorems.
\newblock {\em Discrete Mathematics}, 341(6):1749--1754, 2018.

\bibitem{katona}
G.~Katona.
\newblock Intersection theorems for systems of finite sets.
\newblock {\em Acta Mathematica Academiae Scientiarum Hungarica},
  15(3-4):329--337, 1964.

\bibitem{katonaekr}
G.~O. Katona.
\newblock A simple proof of the {E}rd{\"o}s-{C}hao {K}o-{R}ado theorem.
\newblock {\em Journal of Combinatorial Theory, Series B}, 13(2):183--184,
  1972.

\bibitem{livingstone}
M.~Livingston.
\newblock An ordered version of the {Erd\H{o}s-Ko-Rado} theorem.
\newblock {\em Journal of Combinatorial Theory, Series A}, 26(2):162 -- 165,
  1979.

\bibitem{meyer}
J.-C. Meyer.
\newblock Quelques problemes concernant les cliques des hypergraphes
  {$h$}-complets et {$q$}-parti {$h$}-complets.
\newblock In {\em Hypergraph Seminar}, pages 127--139. Springer, 1974.

\end{thebibliography}
\end{document}